\documentclass[twoside]{article} 

\usepackage{jmlr2e} 
\usepackage{algorithm}
\usepackage{algorithmic}
\usepackage{hyperref}
\usepackage{natbib}

\usepackage{amsmath,amssymb}
\usepackage{pgfplots}
\usepackage{color}
\usepackage{flushend}
\usepackage{multirow}
\usepackage{rotating}
\usepackage{appendix}
\usepackage{tikz,pgfplots}
\usepackage{subcaption}
\usepackage{bbm}

\input{mysymbol.sty}
\usepackage{needspace}

\newtheorem{assumption}{\hspace{0pt}\bf Assumption}

\begin{document}


\jmlrheading{1}{2000}{1-48}{4/00}{10/00}{Aryan Mokhtari and Alejandro Ribeiro}


\ShortHeadings{Conditional Gradient Method for Stochastic Submodular Maximization}{Mokhtari, Hassani, and Karbasi}
\firstpageno{1}
\title{Conditional Gradient Method for Stochastic\\ Submodular Maximization: Closing the Gap}

%

\name{Aryan Mokhtari$^{\S\dagger}$, Hamed Hassani$^{\star}$, and Amin Karbasi$^{ \ddagger}$ \thanks{This work was done while A. Mokhtari was a Research Fellow at the Simons Institute for the Theory of Computing.}
\address{\normalsize $^\S$Laboratory for Information and Decision Systems, MIT \\
\normalsize$^{\dagger}$Simons Institute for the Theory of Computing, University of California Berkeley\\
\normalsize$^{\star}$Department of Electrical and Systems Engineering, University of Pennsylvania\\
\normalsize$^{\ddagger}$Department of Electrical Engineering and Computer Science, Yale University\\
\normalsize aryanm@mit.edu, hassani@seas.upenn.edu, amin.karbasi@yale.edu}}

\maketitle
\thispagestyle{empty}

%

%

%
%
%

\begin{abstract}
In this paper, we study the problem of \textit{constrained} and \textit{stochastic} continuous submodular maximization. Even though the objective function is not concave (nor convex) and is defined in terms of an expectation, we develop a variant of the conditional gradient method, called \alg, which achieves a \textit{tight} approximation guarantee. More precisely, for a monotone and continuous DR-submodular function and subject to a \textit{general} convex body constraint, we prove that \alg achieves a $[(1-1/e)\text{OPT} -\eps]$ guarantee (in expectation) with   $\mathcal{O}{(1/\eps^3)}$ stochastic gradient computations. This guarantee matches the known hardness results and closes the gap between deterministic and stochastic continuous submodular maximization. By using stochastic continuous optimization as an interface, we also provide the first $(1-1/e)$ tight approximation guarantee for maximizing  a \textit{monotone but stochastic} submodular \textit{set} function subject to a general matroid constraint. 
\end{abstract}

\section{Introduction}
Many procedures in statistics  and artificial intelligence require solving  non-convex problems, including clustering \citep{abbasi2007survey}, training deep neural networks \citep{bengio2007greedy}, and performing Bayesian optimization \citep{snoek2012practical}, to name a few.  Historically, the focus has been  to convexify  non-convex objectives; in recent years, there has been significant progress to optimize non-convex functions directly. This direct approach has led to provably good guarantees, for specific problem instances. Examples include latent variable models \citep{anandkumar2014tensor}, non-negative matrix factorization \citep{arora2012computing}, robust PCA \citep{netrapalli2014non}, matrix completions \citep{ge2016matrix}, and training certain specific forms of neural networks \citep{mei2016landscape}. However, it is well known that in general finding the global optimum of a non-convex optimization problem is NP-hard \citep{murty1987some}. This computational barrier has mainly shifted the goal of non-convex optimization towards  two directions: a) finding an \textit{approximate} local minimum by  avoiding saddle points \citep{ge2015escaping,anandkumar2016efficient, jin2017escape, paternain2017second}, or b) characterizing general conditions under which the underlying non-convex optimization is tractable \citep{hazan2016graduated}. 

In this paper, we consider a  broad class of non-convex optimization problems that possess special combinatorial structures. More specifically, we focus on constrained maximization of stochastic continuous submodular functions (CSF) that  demonstrate diminishing returns, i.e., continuous DR-submodular functions,
\begin{equation}\label{eq:mainproblem}
\max_{\bbx \in \ccalC}\ F(\bbx) \doteq \max_{\bbx \in \ccalC}\ \mathbb{E}_{\bbz\sim P}[{\tilde{F} (\bbx,\bbz)}].
\end{equation}
Here, the functions $\tilde{F}: \ccalX \times \mathcal{Z} \to \reals_{+} $ are stochastic where $\bbx \in \ccalX $ is the optimization variable, $\bbz\in  \mathcal{Z}$ is a realization of the random variable $\bbZ$ drawn from a distribution $P$, and $\ccalX\in \reals^n_+$ is a compact set. Our goal is to maximize the expected value of the random functions $\tilde{F}(\bbx,\bbz)$ over the convex body $\ccalC \subseteq \reals^n_+$. Note that we \emph{only assume} that $F(\bbx)$ is DR-submodular, and \emph{not} necessarily the stochastic functions $\tilde{F}(\bbx,\bbz)$. We also consider situations where  the distribution $P$ is either unknown (e.g., when the objective is given as an implicit stochastic model) or the domain of the random variable $\bbZ$ is very large (e.g., when the objective is defined in terms of an empirical risk) which makes the cost of computing the expectation very high. In these regimes, stochastic optimization methods, which operate on computationally cheap estimates of gradients, arise as  natural solutions. In fact, very recently, it was shown in \citep{hassani2017gradient} that stochastic gradient methods achieve a $(1/2)$ approximation guarantee to Problem~\eqref{eq:mainproblem}. The authors also showed  that current versions of the conditional gradient method (a.k.a., Frank-Wolfe), such as continuous greedy \citep{vondrak2008optimal} or its close variant \citep{bian16guaranteed}, can perform arbitrarily poorly in stochastic continuous submodular maximization settings. 

\vspace{2mm}
\textbf{Our contributions.} We provide the first tight $(1-1/e)$ approximation guarantee for Problem~\eqref{eq:mainproblem}  when the continuous  function $F$  is monotone, smooth,   DR-submodular, and the constraint set $\ccalC$ is a bounded convex body. To this end,  we develop a novel conditional gradient method, called  \alg (SCG), that produces a solution with an objective value larger than $((1-1/e)\text{OPT}-\epsilon)$ after $O\left({1}/{\epsilon^3}\right)$ iterations while only having  access to unbiased estimates of the gradients (here OPT denotes the optimal value of Problem~\eqref{eq:mainproblem}). SCG is also memory efficient in the following sense: in contrast to previously proposed conditional gradient methods in stochastic convex \citep{hazan2016variance} and non-convex~\citep{reddi2016stochastic} settings,  SCG does not require using a minibatch in each step. Instead it simply averages over the stochastic estimates of the previous gradients.  

%
%

\vspace{2mm}
\textbf{Connection to Discrete Problems.} Even though submodularity has been mainly studied in  discrete domains \citep{fujishige2005submodular}, many efficient methods for optimizing  such submodular set functions rely on continuous relaxations either through a multi-linear extension \citep{vondrak2008optimal} (for maximization) or Lovas extension \citep{lovasz1983submodular} (for minimization). In fact, Problem~\eqref{eq:mainproblem} has a discrete counterpart, recently considered in \citep{hassani2017gradient, karimi2017stochastic}:
\begin{equation}\label{eq:stochsub}
\max_{S\in \ccalI}f(S) \doteq \max_{S\in \ccalI} \mbE_{\bbz\sim P} [\tilde{f}(S, \bbz)],
\end{equation}
where the functions $\tilde{f}:2^V \times \mathcal{Z} \rightarrow\reals_+$ are stochastic, $S$ is the optimization set variable defined over a ground set $V$, $\bbz\in \mathcal{Z}$ is the realization of a random variable $\bbZ$ drawn from the distribution $P$, and $\ccalI$ is a general matroid constraint.   Since $P$ is unknown, problem~\eqref{eq:stochsub} cannot be directly solved using the current state-of-the-art techniques. Instead, \cite{hassani2017gradient} showed that by lifting the problem to the continuous domain (via multi-linear relaxation) and using stochastic gradient methods on a continuous relaxation to reach a solution that is within a factor $(1/2)$ of the optimum. Contemporarily, \citep{karimi2017stochastic} used a concave relaxation technique to provide a $(1-1/e)$ approximation for the class of submodular coverage functions. Our work also closes the gap for maximizing the stochastic submodular set maximization, namely, Problem~\eqref{eq:stochsub}, by providing the  first tight $(1-1/e)$ approximation guarantee for general monotone submodular set functions subject to a matroid constraint.  

\medskip\noindent{\bf Notation.\quad} Lowercase boldface $\bbv$ denotes a vector and uppercase boldface $\bbA$ a matrix. We use $\|\bbv\|$ to denote the Euclidean norm of vector $\bbv$. The $i$-th element of the vector $\bbv$ is written as $v_i$ and the element on the i-$th$ row and $j$-th column of the matrix $\bbA$ is denoted by $A_{i,j}$.

\section{Related Work}\label{sec:related_work}
Maximizing a deterministic submodular set function has been extensively studied. The celebrated result of \cite{nemhauser1978analysis} shows that a greedy algorithm achieves a $(1-1/e)$ approximation guarantee for a monotone function subject to a cardinality constraint. It is also known that this result is tight  under reasonable complexity-theoretic assumptions \citep{feige1998threshold}. Recently, variants of the greedy algorithm have been proposed to extend the above result to non-monotone and more general constraints \citep{feige2011maximizing, buchbinder2015tight, buchbinder2014submodular, feldman2017greed}. While discrete greedy algorithms are fast, they usually do not provide the tightest guarantees for many classes of feasibility constraints. This is why continuous relaxations of submodular functions, e.g., the multilinear extension, have gained a lot of interest \citep{vondrak2008optimal, calinescu2011maximizing, chekuri2014submodular, feldman2011unified, gharan2011submodular, sviridenko2017optimal}. In particular, it is known that the continuous greedy algorithm achieves a $(1-1/e)$ approximation guarantee for monotone submodular functions under a general matroid constraint \citep{calinescu2011maximizing}. An improved   $((1-e^{-c})/c)$-approximation guarantee can be obtained if $f$ has curvature $c$ \citep{vondrak2010submodularity}. 
%

Continuous submodularity naturally arises   in many learning applications such as robust budget allocation \citep{staib2017robust,soma2014optimal},  online resource allocation \citep{eghbali2016designing},  learning assignments \citep{golovin2014online}, as well as Adwords for e-commerce and advertising \citep{devanur2012online, mehta2007adwords}.  Maximizing a \textit{deteministic} continuous submodular function dates back to the work of \cite{wolsey1982analysis}. More recently, \cite{chekuri2015multiplicative} proposed a multiplicative weight update algorithm that achieves $(1-1/e-\epsilon)$ approximation guarantee after $\tilde{O}(n/\epsilon^2)$ oracle calls to gradients of a monotone  smooth submodular function $F$ (i.e., twice differentiable DR-submodular) subject to a polytope constraint. A similar approximation factor can be obtained after $\mathcal{O}(n/\epsilon)$ oracle calls to gradients of $F$ for monotone DR-submodular functions subject to a down-closed convex body using the continuous greedy method \citep{bian16guaranteed}. {However, such results require exact computation of the gradients $\nabla F$ which is not feasible in Problem~\eqref{eq:mainproblem}. An alternative approach is then to modify the current algorithms by replacing gradients $\nabla F(\bbx_t)$ by their stochastic estimates $\nabla \tilde{F}(\bbx_t,\bbz_t)$; however, this modification may lead to arbitrarily poor solutions as demonstrated in~\citep{hassani2017gradient}. Another alternative is to estimate the gradient by averaging over a (large) mini-batch of samples at each iteration. While this approach can potentially reduce the noise variance,  it increases the computational complexity of each iteration and is not favorable. The work by \cite{hassani2017gradient} is perhaps the first attempt to solve Problem~\eqref{eq:mainproblem} only by executing stochastic estimates of gradients (without using a large batch). They showed that the stochastic gradient ascent method achieves a $(1/2-\epsilon)$ approximation guarantee after $O(1/\epsilon^2)$ iterations. Although this work opens the door for maximizing stochastic CSFs using computationally cheap stochastic gradients, it fails to achieve the optimal $(1-1/e)$ approximation. To close the gap, we propose in this paper \alg which outputs a solution with function value at least  $((1-1/e)\text{OPT}- \epsilon)$ after $O(1/\epsilon^3)$  iterations. Notably, our result only requires the expected function $F$ to be monotone and DR-submodular and the stochastic functions $\tilde{F}$ need not be monotone nor DR-submodular. Moreover, in contrast to the result in \citep{bian16guaranteed}, which holds only for down-closed convex constraints, our result holds for any convex constraints.}

Our result also has important implications for Problem~\eqref{eq:stochsub}; that is, maximizing a stochastic discrete  submodular function subject to a matroid constraint. Since the proposed SCG method works in stochastic settings, we can relax the discrete objective function $f$ in Problem~\eqref{eq:stochsub} to a continuous function $F$ through the multi-linear extension (note that expectation is a linear operator). Then  we can maximize $F$ within a $(1-1/e-\epsilon)$ approximation to the optimum value by using only ${\mathcal{O}}(1/\eps^3)$ oracle calls to the stochastic gradients of $F$. Finally, a proper rounding scheme (such as the contention resolution method \citep{chekuri2014submodular}) results in a feasible set whose value is a $(1-1/e)$ approximation to the optimum set in expectation. 

%
The focus of our paper is on the maximization of stochastic submodular functions. However, there are also very interesting results for minimization of such functions \citep{staib2017robust, ene2017decomposable,chakrabarty2016subquadratic,iyer2013fast}. 
\section{Continuous Submodularity} \label{defs}
We begin by recalling the definition of a submodular set function: A function $f:2^V\rightarrow \reals_+$, defined on the ground set $V$,  is called submodular if for all subsets $A,B\subseteq V$, we have $$f(A)+f(B)\geq f(A\cap B) + f(A\cup B).$$
The notion of submodularity goes beyond the discrete domain \citep{wolsey1982analysis, vondrak2007submodularity, bach2015submodular}.
Consider a continuous function $F: \ccalX \to \reals_{+}$ where the set $\ccalX$ is of  the form $\ccalX=\prod_{i=1}^n\ccalX_i$ and each $\ccalX_i$ is a compact subset of $\reals_+$. We call the continuous function $F$ submodular if for all $\bbx,\bby\in \ccalX$ we have
\begin{align}\label{eq:submodular_def}
F(\bbx) + F(\bby) \geq F(\bbx \vee	 \bby) + F(\bbx \wedge \bby) ,
\end{align}
where $\bbx \vee \bby := \max (\bbx ,\bby )$ (component-wise) and $\bbx \wedge \bby := \min (\bbx ,\bby )$ (component-wise). 
In this paper, our focus is on differentiable  continuous submodular functions with two additional  properties:   monotonicity and diminishing returns.  Formally, a submodular function $F$ is monotone (on the set $\ccalX$) if 
\begin{align}\label{eq:monotone_def}
\bbx \leq \bby  \quad \Longrightarrow \quad F(\bbx) \leq  F(\bby),
\end{align}
for all $\bbx,\bby\in \ccalX$. Note that $\bbx \leq \bby$  in \eqref{eq:monotone_def} means that $x_i\leq y_i$ for all $i=1,\dots,n$. Furthermore, a differentiable submodular function $F$ is called \textit{DR-submodular} (i.e., shows diminishing returns) if the gradients are antitone, namely, for all $\bbx,\bby\in \ccalX$ we have 
\begin{align}\label{eq:antitone_def}
\bbx \leq \bby  \quad \Longrightarrow \quad \nabla F(\bbx) \geq \nabla  F(\bby).
\end{align}
When the function $F$ is twice differentiable, submodularity implies that all cross-second-derivatives are non-positive \citep{bach2015submodular}, i.e., 
\begin{equation}
\forall\ i\neq j,\ \ \forall\ \bbx\in \ccalX, ~~ \frac{\partial^2 F(\bbx)}{\partial x_i \partial x_j} \leq 0, 
\end{equation}
 and DR-submodularity implies that all second-derivatives  are  non-positive \citep{bian16guaranteed}, i.e., 
\begin{equation}
\forall\ i,j,\ \ \forall\ \bbx\in \ccalX, ~~ \frac{\partial^2 F(\bbx)}{\partial x_i \partial x_j} \leq 0.
\end{equation}

\section{Stochastic Continuous Greedy}\label{sec:scg}
In this section, we introduce our main algorithm, \alg (SCG), which  is a stochastic variant of the continuous greedy method to to solve Problem~\eqref{eq:mainproblem}.
We only assume that the expected objective function $F$ is monotone and DR-submodular and the stochastic functions $\tilde{F}(\bbx,\bbz)$ may not be monotone nor submodular. Since the objective function $F$ is monotone and DR-submodular, continuous greedy algorithm \citep{bian16guaranteed,calinescu2011maximizing}  can be used in principle to solve Problem~\eqref{eq:mainproblem}. Note that each update of continuous greedy requires computing the gradient of  $F$, i.e., $\nabla F(\bbx):=\mathbb{E}[{\nabla \tilde{F} (\bbx,\bbz)}]$. However, if we only have access to the (computationally cheap) stochastic gradients ${\nabla \tilde{F} (\bbx,\bbz)}$, then the continuous greedy method will not be directly usable \citep{hassani2017gradient}.  
This limitation is due to the non-vanishing variance of gradient approximations. To resolve this issue, we introduce  stochastic version of the continuous greedy algorithm which reduces the noise of gradient approximations via a common averaging technique in stochastic optimization \citep{ruszczynski1980feasible,ruszczynski2008merit,yang2016parallel,mokhtari2017large}. 

%

Let $t \in  \mathbf{N}$ be a discrete time index and $\rho_t$ a given stepsize which approaches zero as $t$ approaches infinity. Our proposed estimated gradient $\bbd_t$ is defined by the following recursion  
\begin{equation} \label{eq:von1}
\bbd_t = (1-\rho_t) \bbd_{t-1} + \rho_t \nabla \tilde{F}(\bbx_t,\bbz_t),
\end{equation}
where the initial vector is defined as $\bbd_0=\bb0$. It can be shown that the averaging technique in \eqref{eq:von1} reduces the noise of gradient approximation as time increases. More formally, the expected noise of gradient estimation $\E{\|\bbd_t-\nabla F(\bbx_t)\|^2}$ approaches zero asymptotically (Lemma~\ref{lemma:bound_on_grad_approx_sublinear}). This property implies that the gradient estimate $\bbd_t$ is a better candidate for approximating the gradient $\nabla F(\bbx_t)$ comparing to the the unbiased gradient estimate $\nabla \tilde{F}(\bbx_t,\bbz_t)$ that suffers from a high variance approximation. We therefore define the  ascent direction $\bbv_t$ of our proposed SCG method as follows
%
%
\begin{equation}\label{eq:von2}
\bbv_t =\argmax_{\bbv\in \ccalC} \{ \bbd_t^T\bbv \},
\end{equation}
which is a linear objective maximization over the convex set $\ccalC$. Indeed, if instead of the gradient estimate $\bbd_t$ we use the exact gradient $\nabla F(\bbx_t)$ for the updates in~\eqref{eq:von2}, the continuous greedy update will be recovered. Here, as in continuous greedy, the initial decision vector is the null vector, $\bbx_0=\bb0$. Further, the stepsize for updating the iterates is equal to $1/T$, and the variable $\bbx_t$ is updated as
\begin{equation}\label{eq:von3}
\bbx_{t+1} =  \bbx_{t} + \frac{1}{T} \bbv_t.
\end{equation}
The stepsize $1/T$ and the initialization $\bbx_0=\bb0$ ensure that after $T$ iterations the variable $\bbx_T$ ends up in the convex set $\ccalC$. We should highlight that the convex body $\ccalC$ may not be down-closed or contain $\bb0$. Nonetheless, the solution $\bbx_T$ returned by SCG will be a feasible point in $\ccalC$. 
The steps of the proposed SCG method are outlined in Algorithm \ref{algo_SCGGA}.

%
\begin{algorithm}[tb]
\caption{\alg (SCG)}\label{algo_SCGGA} 
\begin{algorithmic}[1] 
{\REQUIRE Stepsizes $\rho_t>0$. Initialize $\bbd_0=\bbx_0=\bb0$
\FOR {$t=1,2,\ldots, T$}
   \STATE Compute $\bbd_t = (1-\rho_t) \bbd_{t-1} + \rho_t \nabla \tilde{F}(\bbx_t,\bbz_t)$;
   \STATE Compute $\bbv_t =\argmax_{\bbv\in \ccalC} \{ \bbd_t^T\bbv \}$;
   \STATE Update the variable $\bbx_{t+1} =\bbx_{t} + \frac{1}{T} \bbv_t$;
\ENDFOR}
\end{algorithmic}\end{algorithm}

\section{Convergence Analysis}\label{sec:cnvg_anal}

In this section, we study the convergence properties of our proposed SCG method for solving Problem~\eqref{eq:mainproblem}. To do so, we first assume that the following conditions hold.

\begin{assumption}\label{ass:bounded_set}
{The Euclidean norm of the elements in the constraint set  \ccalC are uniformly bounded, i.e., for all $\bbx \in \ccalC$ we can write}
\begin{equation}
\|\bbx\|\leq D.
\end{equation}
\end{assumption}
\begin{assumption}\label{ass:smoothness}
The function $F$ is DR-submodular and monotone. Further, its gradients are $L$-Lipschitz continuous over the set $\ccalX$, i.e., for all $\bbx,\bby \in \ccalX$
\begin{equation}
\| \nabla F(\bbx) -  \nabla F(\bby) \| \leq L \| \bbx - \bby \|.
\end{equation}
\end{assumption}
\begin{assumption}\label{ass:bounded_variance}
The variance of the unbiased stochastic gradients $\nabla \tilde{F}(\bbx,\bbz)$ is bounded above by $\sigma^2$, i.e., for any vector $\bbx\in\ccalX$ we can write 
\begin{equation}
\E{\|  \nabla \tilde{F}(\bbx,\bbz) - \nabla F(\bbx)  \|^2} \leq \sigma^2,
\end{equation}
where the expectation is with respect to the randomness of $\bbz \sim P$.
\end{assumption}

Due to the initialization step of  SCG (i.e., starting from $\bb0$) we need a bound on the furthest feasible solution from $\bb0$ that we can end up with; and such a bound is guaranteed by Assumption \ref{ass:bounded_set}.
%
The condition in Assumption~\ref{ass:smoothness} ensures that the objective function $F$ is smooth. Note  again that $\nabla \tilde{F}(\bbx,\bbz)$ may or may not be Lipschitz continuous. 
%
Finally, the required condition in Assumption~\ref{ass:bounded_variance} guarantees that the variance of stochastic gradients $\nabla\tilde{F}(\bbx,\bbz)$ is bounded by a finite constant $\sigma^2<\infty$ which is customary in stochastic optimization.

To study the convergence of SCG, we first derive an upper bound for the expected error of gradient approximation (i.e., $\mathbb{E}[\|\nabla F(\bbx_t) - \bbd_t\|^2]$) in the following lemma. 

\begin{lemma}\label{lemma:bound_on_grad_approx_greedy}
Consider \alg (SCG)  outlined in Algorithm~\ref{algo_SCGGA}. If Assumptions \ref{ass:bounded_set}-\ref{ass:bounded_variance} are satisfied, then the sequence of expected squared gradient errors $\E{\|\nabla F(\bbx_{t}) - \bbd_{t}\|^2}$ for the iterates generated by SCG satisfies
\begin{align}\label{eq:grad_error_bound}
\E{\|\nabla F(\bbx_{t}) - \bbd_{t}\|^2}\leq \left(1-\frac{\rho_t}{2}\right)\E{\|\nabla F(\bbx_{t-1}) - \bbd_{t-1}\|^2}+ \rho_t^2\sigma^2 
+\frac{L^2D^2}{T^2}
+\frac{2L^2D^2 }{\rho_t T^2}.
\end{align}
\end{lemma}

\begin{proof}
See Section \ref{proof:lemma:bound_on_grad_approx_greedy}.
\end{proof}

The result in Lemma \ref{lemma:bound_on_grad_approx_greedy} showcases that the expected squared error of gradient approximation $\E{\|\nabla F(\bbx_t) - \bbd_t\|^2}$ decreases at each iteration by the factor $(1-\rho_t/2)$ if the remaining terms on the right hand side of \eqref{eq:grad_error_bound} are negligible relative to the term $(1-\rho_t/2) \mathbb{E}[\|\nabla F(\bbx_{t-1}) - \bbd_{t-1}\|^2]$. This condition can be satisfied, if the parameters $\{\rho_t\}$ are chosen properly. We formalize this claim in the following lemma and show that the expected error $\mathbb{E}[{\|\nabla F(\bbx_{t}) - \bbd_{t}\|^2}]$ converges to zero at a sublinear rate of $\mathcal{O}(t^{-2/3})$.

\begin{lemma}\label{lemma:bound_on_grad_approx_sublinear}
Consider \alg (SCG)  outlined in Algorithm~\ref{algo_SCGGA}.  If  Assumptions \ref{ass:bounded_set}-\ref{ass:bounded_variance} are satisfied and $\rho_t=\frac{4}{(t+8)^{2/3}}$, then for $t=0,\dots,T$ we have
\begin{align}\label{eq:grad_error_bound_2}
\E{ {\|\nabla F(\bbx_{t}) - \bbd_{t}\|^2} }&\leq \frac{Q}{(t+9)^{2/3}},
\end{align}
where $Q:=\max \{ \|\nabla F(\bbx_{0}) - \bbd_{0}\|^2 9^{2/3} , 16\sigma^2+3L^2 D^2 \}.$
\end{lemma}

\begin{proof}
See Section \ref{proof:lemma:bound_on_grad_approx_sublinear}.
\end{proof}

Let us now use the result in Lemma~\ref{lemma:bound_on_grad_approx_sublinear} to show that the sequence of iterates generated by SCG reaches  a $(1-1/e)$ approximation  for Problem~\eqref{eq:mainproblem}.

\begin{theorem}\label{thm:optimal_bound_greedy}
Consider \alg (SCG)  outlined in Algorithm~\ref{algo_SCGGA}.  If  Assumptions \ref{ass:bounded_set}-\ref{ass:bounded_variance} are satisfied and $\rho_t=\frac{4}{(t+8)^{2/3}}$, then the expected objective function value for the iterates generated by SCG satisfies the inequality
\begin{align}\label{eq:claim_for_sto_greedy}
\E{F(\bbx_T)} \geq (1-1/e) \text{OPT}- \frac{2DQ^{1/2}}{T^{1/3}}-  \frac{LD^2}{2T},
\end{align}
where $\text{OPT}\:=\max_{\bbx \in \ccalC}\ F(\bbx)$.
\end{theorem}

\begin{proof}
Let $\bbx^*$ be the global maximizer within the constraint set $\mathcal{C}$. Based on the smoothness of the function $F$ with constant $L$ we can write 
\begin{align}\label{proof:final_result_100}
F(\bbx_{t+1})
& \geq F(\bbx_{t}) + \langle \nabla F(\bbx_t), \bbx_{t+1}-\bbx_t \rangle - \frac{L}{2}  || \bbx_{t+1}-\bbx_t||^2 \nonumber\\
& = F(\bbx_{t}) + \frac{1}{T} \langle \nabla F(\bbx_t),\bbv_{t} \rangle - \frac{L}{2T^2} || \bbv_{t} ||^2,
\end{align}
where the equality follows from the update in \eqref{eq:von3}. Since  $\bbv_t$ is in the set $\ccalC$, it follows from Assumption \ref{ass:bounded_set} that the norm $\|\bbv_t\|^2$ is bounded above by $D^2$. Apply this substitution and add and subtract the inner product $\langle \bbd_t,\bbv_t   \rangle$ to the right hand side of \eqref{proof:final_result_100} to obtain 
\begin{align}\label{proof:final_result_200}
 F(\bbx_{t+1}) 
&\geq  F(\bbx_{t}) + \frac{1}{T} \langle \bbv_{t}, \bbd_t \rangle +  \frac{1}{T} \langle \bbv_{t}, \nabla F(\bbx_t) - \bbd_t \rangle - \frac{LD^2}{2T^2} \nonumber\\
&\geq F(\bbx_{t}) +\frac{1}{T}\langle \bbx^*, \bbd_t \rangle + \frac{1}{T} \langle \bbv_{t}, \nabla F(\bbx_t) - \bbd_t \rangle - \frac{LD^2}{2T^2}.
\end{align}
Note that the second inequality in \eqref{proof:final_result_200} holds since based on \eqref{eq:von2} we can write  $\langle \bbx^*, \bbd_t \rangle \leq \max_{v\in \ccalX} \{ \langle \bbv, \bbd_t \rangle\}  = \langle \bbv_t, \bbd_t \rangle$. Now add and subtract the inner product $ \langle \bbx^*, \nabla F(\bbx_t) \rangle /T$ to the RHS of \eqref{proof:final_result_200} to get
\begin{align}\label{proof:final_result_300}
 F(\bbx_{t+1}) & \geq F(\bbx_{t}) + \frac{1}{T} \langle \bbx^*, \nabla F(\bbx_t) \rangle 
 +  \frac{1}{T} \langle \bbv_{t} - \bbx^*, \nabla F(\bbx_t) - \bbd_t \rangle - \frac{LD^2}{2T^2}.
\end{align}
We further have $ \langle \bbx^*, \nabla F(\bbx_t) \rangle \geq F(\bbx^*) - F(\bbx_t)$; this follows from monotonicity of $F$ as well as concavity of $F$ along positive directions; see, e.g., \citep{calinescu2011maximizing}. Moreover, by Young's inequality we can show that the inner product $\langle \bbv_{t} - \bbx^*, \nabla F(\bbx_t) - \bbd_t \rangle$ is lower bounded by $-  (\beta_t/2)||\bbv_{t} - \bbx^*||^2 - (1/2\beta_t){|| \nabla F(\bbx_t) - \bbd_t ||^2} $ for any $\beta_t>0$. By applying these substitutions into \eqref{proof:final_result_300} we obtain 
\begin{align}\label{proof:final_result_400}
 &F(\bbx_{t+1}) 
\geq F(\bbx_{t}) + \frac{1}{T} (F(\bbx^*) - F(\bbx_{t})) - \frac{LD^2}{2T^2}
-  \frac{1}{2T}\left(\beta_t||\bbv_{t} - \bbx^*||^2 + \frac{|| \nabla F(\bbx_t) - \bbd_t ||^2}{\beta_t}\right).
\end{align}
Replace $||\bbv_{t} - \bbx^*||^2$ by its upper bound $4D^2$ and compute the expected value of \eqref{proof:final_result_400} to write 
\begin{align}\label{proof:final_result_500}
& \E{F(\bbx_{t+1}) } \geq  \E{F(\bbx_{t})} + \frac{1}{T}\E{F(\bbx^*) - F(\bbx_{t}))} 
-  \frac{1}{2T} \left[ 4\beta_t D^2 + \frac{\E{|| \nabla F(\bbx_t) - \bbd_t ||^2}}{\beta_t}\right] - \frac{LD^2}{2T^2}.
\end{align}
Substitute $\E{|| \nabla F(\bbx_t) - \bbd_t ||^2}$ by its upper bound ${Q}/({(t+9)^{2/3}})$ according to the result in \eqref{eq:grad_error_bound_2}. Further, set $\beta_t= (Q^{1/2})/(2D(t+9)^{1/3})$ and regroup the resulted expression to obtain 
\begin{align}\label{proof:final_result_600}
 \E{F(\bbx^*) - F(\bbx_{t+1}) }& \leq \left(1-\frac{1}{T}\right) \E{F(\bbx^*) -F(\bbx_{t})} 
+  \frac{2DQ^{1/2}}{(t+9)^{1/3}T}+\frac{LD^2}{2T^2}.
\end{align}
By applying the inequality in \eqref{proof:final_result_600} recursively for $t=0,\dots,T-1$ we obtain 
\begin{align}\label{proof:final_result_700}
& \E{F(\bbx^*) - F(\bbx_{T}) } \leq \left(1-\frac{1}{T}\right)^T ({F(\bbx^*) -F(\bbx_{0})}  )
+ \sum_{t=0}^{T-1} \frac{2DQ^{1/2}}{(t+9)^{1/3}T}+ \sum_{t=0}^{T-1} \frac{LD^2}{2T^2}.
\end{align}
Simplifying the terms on the right hand side \eqref{proof:final_result_700} leads to the expression 
\begin{align}\label{proof:final_result_800}
& \E{F(\bbx^*) - F(\bbx_{T}) }
  \leq \frac{1}{e} ({F(\bbx^*) -F(\bbx_{0})}  )
+  \frac{2DQ^{1/2}}{T^{1/3}}+ \frac{LD^2}{2T}.
\end{align}
Here, we use the fact that $F(\bbx_{0}) \geq0$, and hence the expression in \eqref{proof:final_result_800} can be simplified to  
\begin{equation}\label{proof:final_result_900}
 \E{F(\bbx_{T}) }\geq (1- 1/e) F(\bbx^*)  - \frac{2DQ^{1/2}}{T^{1/3}}-  \frac{LD^2}{2T},
\end{equation}
and the claim in \eqref{eq:claim_for_sto_greedy} follows. 
\end{proof}

The result in Theorem \ref{thm:optimal_bound_greedy} shows that the sequence of iterates generated by SCG, which only has access to a noisy unbiased estimate of the gradient at each iteration, is able to achieve the optimal approximation bound $(1-1/e)$, while the error term vanishes at a sublinear rate of $\mathcal{O}(T^{-1/3})$.

\section{Discrete Submodular Maximization}

According to the results in Section~\ref{sec:cnvg_anal}, the SCG method achieves in expectation a  $(1-1/e)$-optimal solution for Problem~\eqref{eq:mainproblem}. The focus of this section is on extending this result into the discrete domain and showing that  SCG can be applied for maximizing a stochastic submodular \emph{set} function $f$, namely Problem~\eqref{eq:stochsub}, through the multilinear extension of the function $f$. To be more precise, in lieu of solving the program in~\eqref{eq:stochsub}
%
%
%
%
one can solve the continuous optimization problem
\begin{align}\label{eq:multilinear_program}
\max_{\bbx \in \ccalC} \ F(\bbx),
\end{align}
where $F$ is the multilinear extension of the function $f$ defined as 
\begin{equation}\label{eq:def_multi_linear_extension}
F(\bbx) = \sum_{S\subset V}f(S) \prod_{i\in S} x_i \prod_{j\notin S} (1-x_j) ,
\end{equation}
and the convex set $\ccalC= \text{conv}\{1_{I} : I\in \ccalI \}$ is the matroid polytope \citep{calinescu2011maximizing}. Note that in \eqref{eq:def_multi_linear_extension}, $x_i$ denotes the $i$-th element of the vector $\bbx$.

Indeed, the continuous greedy algorithm is able to solve the program in \eqref{eq:multilinear_program}; however, each iteration of the method is computationally costly due to gradient $\nabla F(\bbx)$ evaluations. Instead, \cite{badanidiyuru2014fast} and \cite{chekuri2015multiplicative} suggested approximating the gradient using a sufficient number of samples from $f$.  This mechanism still requires access to the set function $f$ multiple times at each iteration, and hence is not feasible for solving Problem~\eqref{eq:stochsub}. The idea is then to use a stochastic (unbiased) estimate for the gradient  $\nabla F$. In Appendix~\ref{unbiased}, we provide a method to compute an unbiased estimate of the gradient using $n$ samples from $\tilde{f}(S_i, \bbz)$, where $\bbz \sim P$ and $S_i$'s, $i=1, \cdots, n$, are carefully chosen sets.
Indeed, the stochastic gradient ascent method proposed in \citep{hassani2017gradient} can be used to solve the multilinear extension problem in \eqref{eq:multilinear_program} using unbiased estimates of the gradient at each iteration. However, the stochastic gradient ascent method fails to achieve the optimal $(1-1/e)$ approximation. Further, the work in \citep{karimi2017stochastic} achieves a $(1-1/e)$ approximation solution only when each $\tilde{f}(\cdot, \bbz)$ is a coverage function. Here, we show that SCG achieves the first $(1-1/e)$ tight approximation guarantee for the discrete stochastic submodular Problem~\eqref{eq:stochsub}. More precisely, we show  that SCG finds a solution for \eqref{eq:multilinear_program}, with an expected function value that is at least $(1-1/e)\text{OPT} -\epsilon$, in $\mathcal{O}(1/\epsilon^3)$ iterations.   
To do so, we first show in the following lemma  that the difference between any coordinates of gradients of two consecutive iterates generated by SCG, i.e., $\nabla_j F(\bbx_{t+1})-\nabla_j F(\bbx_{t})$ for $j\in\{1,\dots,n\}$, is bounded by $\|\bbx_{t+1}-\bbx_{t}\|$ multiplied by a factor which is independent of the problem dimension $n$.

\begin{lemma}\label{lemma:lip_constant}
Consider \alg (SCG)  outlined in Algorithm~\ref{algo_SCGGA} with iterates $\bbx_t$, and recall the definition of the multilinear extension function $F$ in~\eqref{eq:def_multi_linear_extension}. If we define $r$ as the rank of the matroid $\ccalI$ and $m_f \triangleq \max_{i \in \{1, \cdots, n\}} f(i)$, then 
\begin{align}\label{claim:lip_constant}
\left|\nabla_{j} F(\bbx_{t+1}) -\nabla_{j} F(\bbx_t) \right| \leq m_f\sqrt{r} \|\bbx_{t+1}-\bbx_{t}\|,
\end{align}
holds for $j=1,\dots,n$.
\end{lemma}

\begin{proof}
See Section \ref{proof:lemma:lip_constant}.
\end{proof}

The result in Lemma \ref{lemma:lip_constant} states that in an \textit{ascent direction of SCG}, the gradient is $m_f\sqrt{r}$-Lipschitz continuous. 
Here, $m_f$ is the maximum marginal value of the function $f$ and $r$ is the rank for the matroid. 
Using the result of Lemma~\ref{lemma:lip_constant} and a coordinate-wise analysis, the bounds in Theorem~\ref{thm:optimal_bound_greedy} can be improved and specified for the case of multilinear extension maximization problem as we show in the following theorem.

\begin{theorem}\label{thm:multi_linear_extenstion_thm}
Consider \alg (SCG)  outlined in Algorithm~\ref{algo_SCGGA}. Recall the definition of the multilinear extension function $F$ in~\eqref{eq:def_multi_linear_extension} and the definitions of $r$ and $m_f$ in Lemma \ref{lemma:lip_constant}. Further, set the averaging parameter as $\rho_t=4/(t+8)^{2/3}$.  If Assumptions~\ref{ass:bounded_set} and \ref{ass:bounded_variance} hold, then the iterate $\bbx_T$ generated by SCG satisfies the inequality
\begin{align}\label{eq:claim_for_sto_greedy_multi_linear}
\E{F(\bbx_T)} \geq (1-1/e) OPT- \frac{2DK}{T^{1/3}}
 -\frac{m_f\sqrt{r}D^2}{2T},
\end{align}
where $K:=\max \{ \|\nabla F(\bbx_{0}) - \bbd_{0}\| 9^{1/3} , 4\sigma+\sqrt{3r}m_f D \}.$
\end{theorem}

\begin{proof}
The proof is similar to the proof of Theorem~ \ref{thm:optimal_bound_greedy}. The main difference is to write the analysis coordinate-wise and replace $L$ by  ${m_f\sqrt{r}}$, as shown in Lemma \ref{lemma:lip_constant}. For more details, check Section \ref{proof:thm:multi_linear_extenstion_thm} in the supplementary material. 
\end{proof}

The result in Theorem \ref{thm:multi_linear_extenstion_thm} indicates that the sequence of iterates generated by SCG achieves a $(1-1/e) \text{OPT} - \eps$ approximation guarantee. Note that the constants on the right hand side of \eqref{eq:claim_for_sto_greedy_multi_linear} are independent of $n$, except $K$ which depends on $\sigma$. It can be shown that, in the worst case, the variance  $\sigma$ depends on the size of the ground set $n$ and the variance of the stochastic functions $\tilde{f}(\cdot,\bbz)$.

Let us now explain how the variance of the stochastic gradients of $F$ relates to the variance of the marginal values of $f$. Consider a generic submodular set function $g$ and its multilinear extension $G$.  It is easy to show that
\begin{equation}
\nabla_j G(\bbx) = G(\bbx; x_j=1) - G(\bbx; x_j=0).
\end{equation}
 Hence, from 
submodularity we have $\nabla_j G(\bbx) \leq g(\{j\})$. Using this simple fact we can deduce that  
\begin{equation}
\E{\|  \nabla \tilde{F}(\bbx,\bbz) - \nabla F(\bbx)  \|^2} \leq n\max_{j  \in [n]} \mathbb{E}[ \tilde{f}(\{j\}, \mathbf{z})^2 ].
\end{equation}
Therefore, the constat $\sigma $ can be upper bounded by 
\begin{equation}
\sigma\leq \sqrt{n}\max_{j  \in [n]} \mathbb{E}[ \tilde{f}(\{j\}, \mathbf{z})^2 ]^{1/2}.
\end{equation}
 As a result, we have the following guarantee for SCG in the case of multilinear functions.

\begin{corollary}
Consider \alg (SCG)  outlined in Algorithm~\ref{algo_SCGGA}. Suppose the conditions in Theorem \ref{thm:multi_linear_extenstion_thm} are satisfied. Then, the sequence of iterates generated by SCG achieves a $(1-1/e)OPT - \epsilon$ solution after $\mathcal{O}({n^{3/2}}/{\eps^3})$ iterations.
\end{corollary}

\begin{proof}
According to the result in Theorem \ref{thm:multi_linear_extenstion_thm}, SCG reaches a $(1-1/e)OPT-\mathcal{O}(n^{1/2}/T^{1/3})$ solution after $T$ iterations. Therefore, to achieve a $((1-1/e)OPT-\epsilon)$ approximation, $\mathcal{O}(n^{3/2}/\eps^{3})$ iterations are required. 
\end{proof}

\section{Numerical Experiments}

In our experiments, we consider a movie recommendation application \citep{serban17} consisting of $N$ users and $n$ movies. Each user $i$ has a user-specific utility function $f(\cdot,i)$ for evaluating sets of movies. The goal is to find a set of $k$ movies such that in expectation over users' preferences it provides the highest utility, i.e., $\max_{|S|\leq k}f(S)$, where $f(S)  \doteq\mathbb{E}_{i\sim P}[f(S,i)]$.  This is an instance of the (discrete) stochastic submodular maximization problem defined in \eqref{eq:stochsub}. For simplicity, we assume $f$ has the form of an empirical objective function, i.e. $f(S) = \frac{1}{N}\sum_{i=1}^N f(S,i)$. In other words, the distribution $P$ is assumed to be uniform on the integers between $1$ and $N$. 
 The continuous counterpart of this problem is to consider the  the multilinear extension $F(\cdot,i)$ of any function $f(\cdot,i)$ and solve the problem in the continuous domain as follows. Let $F(\bbx) = \mathbb{E}_{i \sim \mathcal{D}} [F(\bbx,i)]$ for $x \in [0,1]^n$ and define the constraint set $\mathcal{C} =   \{ \bbx \in [0,1]^N: \sum_{i=1}^n x_i \leq k\}$. The discrete and continuous optimization formulations lead to the same optimal value \citep{calinescu2011maximizing}: 
 \begin{equation}
 \max_{S: |S| \leq k} f(S) = \max_{\bbx \in \mathcal{C}} F(\bbx).
 \end{equation}
 Therefore, by running SCG we can find a solution in the continuous domain that is at least $1-1/e$ approximation to the optimal value. By rounding that fractional solution (for instance via randomized Pipage rounding~\citep{calinescu2011maximizing}) we obtain a set whose utility is at least $1-1/e$ of the optimum solution set of size $k$. We note that randomized Pipage rounding does not need access to the value of $f$. We also remark that each iteration of SCG can be done very efficiently in $O(n)$ time (the $\argmax$ step reduces to finding the largest $k$ elements of a vector of length $n$). Therefore, such approach easily scales to big data scenarios where the size of the data set $N$ (e.g.~number of users) or the number of items $n$ (e.g.~number of movies) are very large.

\noindent In our experiments, we consider the following baselines: 
\begin{itemize}
\item[(i)] Stochastic Continuous Greedy (SCG) with $\rho_t = \frac 12 t^{-2/3}$ and mini-batch size $B$. The details for computing an unbiased estimator for the gradient of $F$ are given in Section \ref{unbiased} in the supplementary material.
  \item[(ii)] Stochastic Gradient Ascent (SGA) of \citep{hassani2017gradient}: with  stepsize $\mu_t = c/\sqrt{t}$ and mini-batch size $B$. 
 \item[(iii)] Frank-Wolfe (FW) variant of \citep{bian16guaranteed,calinescu2011maximizing}: with parameter $T$ for the total number of iterations and batch size $B$  (we further let $\alpha =1, \delta=0$, see Algorithm 1 in \citep{bian16guaranteed} or the continuous greedy method of \citep{calinescu2011maximizing} for more details).  		
 \item [(iv)] Batch-mode Greedy (Greedy): by running the vanilla greedy algorithm (in the discrete domain) in the following way. At each round of the algorithm (for selecting a new element), $B$ random users are picked and the function $f$ is estimated by the average over of the $B$ selected users.  
\end{itemize}

To run the experiments we use the MovieLens data set. It consists of 1 million ratings (from 1 to 5) by $N=6041$ users for $n=4000$ movies.  Let $r_{i,j}$ denote the rating of user $i$ for movie $j$ (if such a rating does not exist we assign $r_{i,j}$ to 0). In our experiments, we consider two well motivated objective functions. The first one is called ``facility location "  where the valuation function by user $i$ is defined as  $f(S,i) = \max_{j\in S} r_{i,j}$. In words, the way user $i$ evaluates a set $S$ is by picking the highest rated movie in $S$.  Thus, the objective function is 
\begin{equation}
f_{\rm fac}(S) = \frac{1}{N}\sum_{i=1}^N  \max_{j\in S} r_{i,j}.
\end{equation}

In our second experiment, we consider a  different user-specific valuation function which is a concave function composed with a modular function, i.e.,  $f(S,i) =(\sum_{j\in S} r_{i,j})^{1/2}.$ Again, by considering the uniform distribution over the set of users, we obtain 
\begin{equation}
f_{\rm con}(S) = \frac{1}{N} \sum_{i=1}^N  \Big(\sum_{j\in S} r_{i,j}\Big)^{1/2}.
\end{equation}

\begin{figure*}[t!]
  \centering
    \begin{subfigure}{0.32\textwidth}
    \begin{center}
      \centerline{\includegraphics[width=1.12\columnwidth]{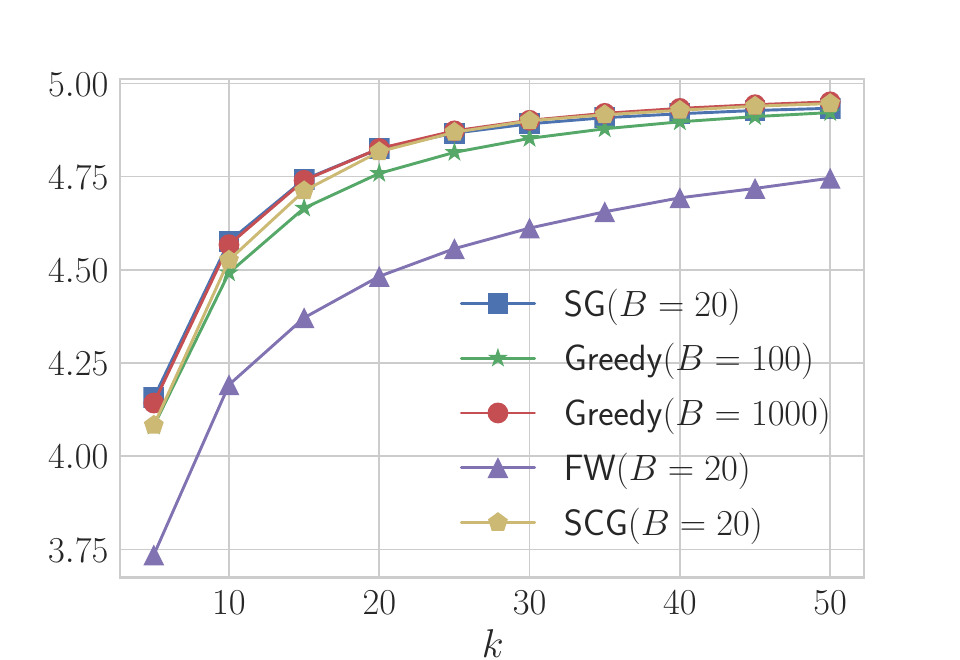}}
            \caption{facility location }
      \label{fig3}
    \end{center}
      \end{subfigure}
        \begin{subfigure}{0.32\textwidth}
    \begin{center}
      \centerline{\includegraphics[width=1.12\columnwidth]{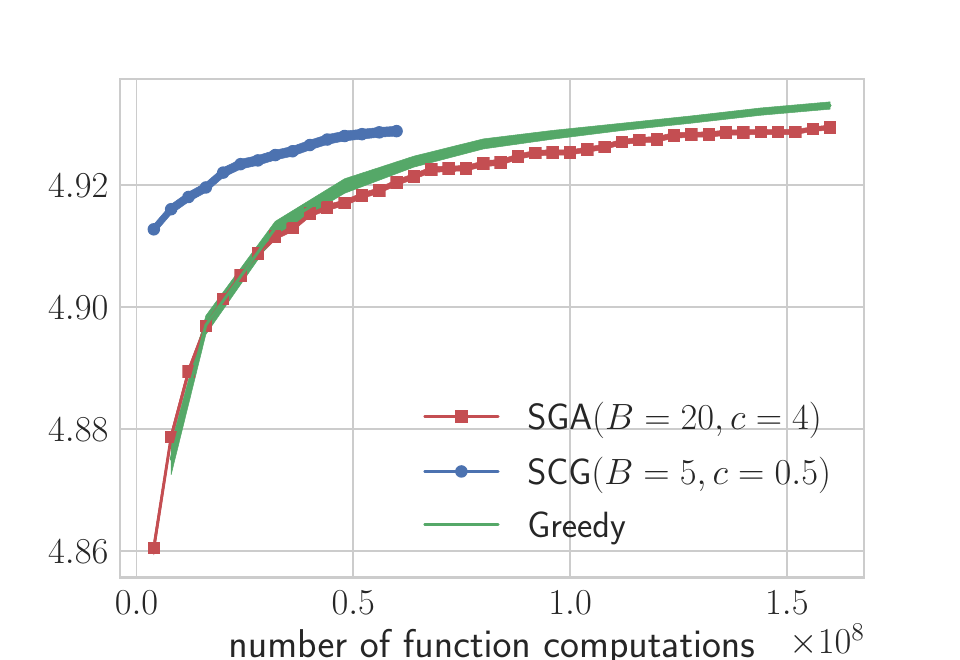}}
            \caption{facility location }
      \label{fig4}
    \end{center}
  \end{subfigure}
    \begin{subfigure}{0.32\textwidth}
    \begin{center}
      \centerline{\includegraphics[width=1.12\columnwidth]{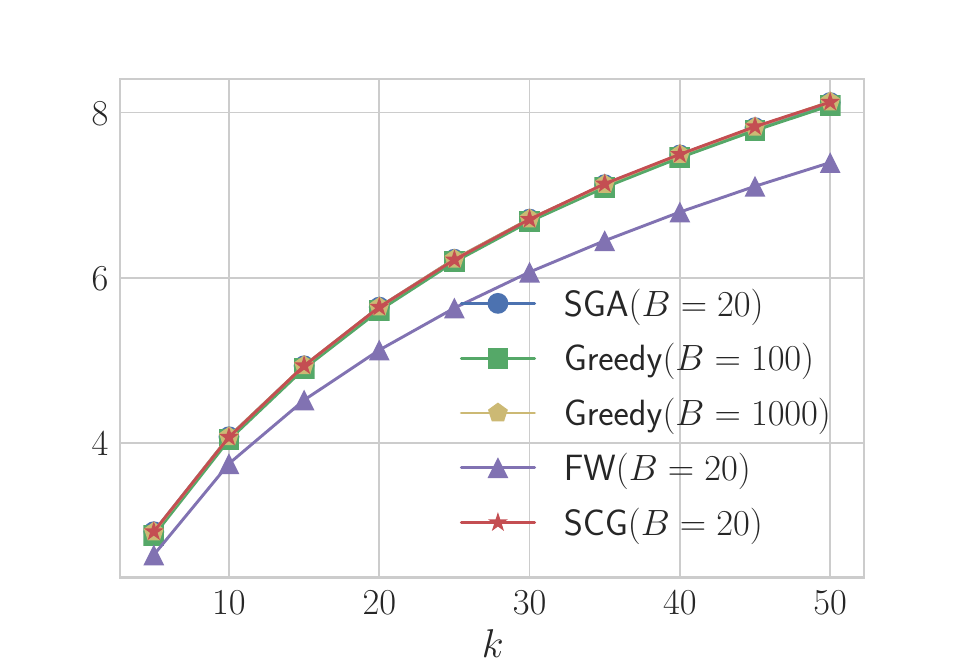}}
      \caption{concave-over-modular}
      \label{fig1}
    \end{center}
  \end{subfigure}%
    \caption{Comparison of the performances of SG, Greedy, FW, and SCG in a movie recommendation application. 
    Fig. \ref{fig3} illustrates the performance of the algorithms in terms of the facility-location objective value w.r.t. the cardinality constraint size~$k$ after $T = 2000$ iterations. Fig. \ref{fig4} compares the considered methods in terms of runtime (for a fixed $k=40$) by illustrating the facility location  objective function value vs. the number of (simple) function evaluations. 
     Fig. \ref{fig1} demonstrates the concave-over-modular objective function value vs. the size of the cardinality constraint $k$ after running the algorithms for $T = 2000$ iterations. 
 \label{fig-kolli}}
 \end{figure*}

Figure~\ref{fig-kolli} depicts the performance of different algorithms for the two proposed objective functions.  As Figures~\ref{fig3} and  \ref{fig1} show, the FW algorithm needs a higher mini-batch size to be comparable to SCG.  Note that a smaller batch size leads to less computational effort (under the same value for $B,T$, the computational complexity of FW, SGA, SCG is almost the same).  Figures~\ref{fig4} shows the performance of the algorithms with respect to the number of times the (simple) functions (i.e., $f(\cdot,i)$'s) are evaluated. Note that the total number of (simple) function evaluations for SGA and SGA is $nBT$, where $T$ is the number of iterations.  Also, for Greedy the total number of evaluations is $nkB$.  This further shows that SCG has a better computational complexity requirement w.r.t. SGA as well as the Greedy algorithm (in the discrete domain).

\section{Conclusion}\label{sec_conclusion}
In this paper, we provided the first tight approximation guarantee for maximizing a stochastic monotone DR-submodular function subject to a general convex body constraint. We developed \alg that achieves a $[(1-1/e)\text{OPT} -\eps]$ guarantee (in expectation) with   $\mathcal{O}{(1/\eps^3)}$ stochastic gradient computations. We also demonstrated that our continuous algorithm  can be used to provide the first $(1-1/e)$ tight approximation guarantee for maximizing  a \textit{monotone but stochastic} submodular \textit{set} function subject to a general matroid constraint. We believe that our results provide an important step towards
unifying discrete and continuous submodular optimization in  stochastic settings. 

%
%
%
%
%
%



\section{Appendix}


\subsection{Proof of Lemma \ref{lemma:bound_on_grad_approx_greedy}}\label{proof:lemma:bound_on_grad_approx_greedy}

Use the definition $\bbd_t := (1-\rho_t) \bbd_{t-1} + \rho_t \nabla \tilde{F}(\bbx_t,\bbz_t)$  to write $\|\nabla F(\bbx_t) - \bbd_t\|^2$ as
\begin{align}\label{proof:bound_on_grad_100}
\|\nabla F(\bbx_t) - \bbd_t\|^2 
=\|\nabla F(\bbx_t) -(1-\rho_t) \bbd_{t-1} - \rho_t \nabla \tilde{F}(\bbx_t,\bbz_t)\|^2.
\end{align}
Add and subtract the term $(1-\rho_t)\nabla F(\bbx_{t-1})$ to the right hand side of \eqref{proof:bound_on_grad_100}, regroup the terms to obtain
\begin{align}\label{proof:bound_on_grad_200}
&\|\nabla F(\bbx_t) - \bbd_t\|^2\nonumber\\
&=\|\rho_t(\nabla F(\bbx_t)-\nabla \tilde{F}(\bbx_t,\bbz_t))  
+(1-\rho_t)(\nabla F(\bbx_t)-\nabla F(\bbx_{t-1}))
+(1-\rho_t)(\nabla F(\bbx_{t-1}) - \bbd_{t-1} )\|^2.
\end{align}
Define $\ccalF_t$ as a sigma algebra that measures the history of the system up until time $t$. Expanding the square and computing the conditional expectation $\E{\cdot\mid \ccalF_t}$ of the resulted expression yield 
\begin{align}\label{proof:bound_on_grad_300}
&\E{\|\nabla F(\bbx_t) - \bbd_t\|^2\mid\ccalF_t} =\rho_t^2\E{\|\nabla F(\bbx_t)-\nabla \tilde{F}(\bbx_t,\bbz_t)\|^2\mid\ccalF_t}+   (1-\rho_t)^2\|\nabla F(\bbx_{t-1}) - \bbd_{t-1} \|^2 
     \nonumber\\
     &  \quad
   +(1-\rho_t)^2\|\nabla F(\bbx_t)-\nabla F(\bbx_{t-1})\|^2
    + 2 (1-\rho_t)^2\langle \nabla F(\bbx_t)-\nabla F(\bbx_{t-1}) , \nabla F(\bbx_{t-1}) - \bbd_{t-1} \rangle.
 \end{align}
  The term $\E{\|\nabla F(\bbx_t)-\nabla \tilde{F}(\bbx_t,\bbz_t)\|^2\mid\ccalF_t}$ can be bounded above by $\sigma^2$ according to Assumption \ref{ass:bounded_variance}. Based on Assumptions \ref{ass:bounded_set} and \ref{ass:smoothness}, we can also show that the squared norm $\|\nabla F(\bbx_t)-\nabla F(\bbx_{t-1})\|^2$ is upper bounded by $L^2D^2/T^2$. Moreover, the inner product $2\langle \nabla F(\bbx_t)\!-\!\nabla F(\bbx_{t-1}) , \nabla F(\bbx_{t-1}) - \bbd_{t-1} \rangle$ can be upper bounded by $\beta_t \|\nabla F(\bbx_{t-1}) - \bbd_{t-1}\|^2+(1/\beta_t) L^2D^2/T^2 $ using Young's inequality (i.e., $2\langle \bba,\bbb\rangle \leq \beta\|\bba\|^2+\|\bbb\|^2/\beta$  for any $\bba,\bbb\in \reals^n$ and $\beta>0$) and the conditions in Assumptions \ref{ass:bounded_set} and \ref{ass:smoothness}, where $\beta_t>0$ is a free scalar. Applying these substitutions into \eqref{proof:bound_on_grad_300} leads to 
\begin{align}\label{proof:bound_on_grad_400}
&\E{\|\nabla F(\bbx_t) - \bbd_t\|^2\mid\ccalF_t} \leq \rho_t^2\sigma^2
   +(1-\rho_t)^2 (1+\frac{1}{\beta_t})\frac{L^2D^2}{T^2}
   +(1-\rho_t)^2(1+\beta_t)\|\nabla F(\bbx_{t-1}) - \bbd_{t-1} \|^2. 
   \end{align}
Replace $(1-\rho_t)^2$ by $(1-\rho_t)$, set $\beta:=\rho_t/2$, and compute the expectation with respect to $\ccalF_0$ to obtain
\begin{align}\label{proof:bound_on_grad_500}
\E{\|\nabla F(\bbx_t) - \bbd_t\|^2}\leq \rho_t^2\sigma^2
   +\frac{ L^2D^2}{T^2} +\frac{2L^2D^2}{\rho_tT^2}
   +\left(1-\frac{\rho_t}{2}\right)\E{\|\nabla F(\bbx_{t-1}) - \bbd_{t-1} \|^2},
\end{align}
and the claim in \eqref{eq:grad_error_bound} follows.


\subsection{Proof of Lemma \ref{lemma:bound_on_grad_approx_sublinear}}\label{proof:lemma:bound_on_grad_approx_sublinear}
Define $a_t:=\E{ {\|\nabla F(\bbx_{t}) - \bbd_{t}\|^2} }$. Also, assume $\rho_t=\frac{4}{(t+s)^{2/3}}$ where $s$ is a fixed scalar and satisfies the condition $8\leq s \leq T$ (so the proof is slightly more general). Apply these substitutions into\eqref{eq:grad_error_bound} to obtain
\begin{align}\label{proof:bound_on_grad_sublinear_100}
a_t
\leq \left(1-\frac{2}{(t+s)^{2/3}}\right)a_{t-1}+ \frac{16\sigma^2}{(t+s)^{4/3}}
   +\frac{ L^2D^2}{T^2} 
   +\frac{L^2D^2(t+s)^{2/3}}{2T^2}.
\end{align}
Now use the conditions $s\leq T$ and $t\leq T$ to replace $1/T$ in \eqref{proof:bound_on_grad_sublinear_100} by its upper bound $2/(t+s)$. Applying this substitution leads to 
\begin{align}\label{proof:bound_on_grad_sublinear_200}
a_t
\leq \left(1-\frac{2}{(t+s)^{2/3}}\right)a_{t-1}+ \frac{16\sigma^2}{(t+s)^{4/3}}
   +\frac{4 L^2D^2}{(t+s)^2} 
   +\frac{2L^2D^2}{(t+s)^{4/3}}.
\end{align}
Since $t+s\geq 8$ we can write $(t+s)^2=(t+s)^{4/3} (t+s)^{2/3}\geq (t+s)^{4/3} 8^{2/3}\geq 4(t+s)^{4/3}$. Replacing the term $(t+s)^2$ in \eqref{proof:bound_on_grad_sublinear_200} by $4(t+s)^{4/3}$ and regrouping the terms lead to 
\begin{align}\label{proof:bound_on_grad_sublinear_300}
a_t
\leq \left(1-\frac{2}{(t+s)^{2/3}}\right)a_{t-1}+ \frac{16\sigma^2+3L^2 D^2}{(t+s)^{4/3}}
\end{align}
Now we prove by induction that for $t=0,\dots,T$ we can write 
\begin{equation}\label{proof:bound_on_grad_sublinear_400}
a_t\leq \frac{Q}{(t+s+1)^{2/3}}, 
\end{equation}
where $Q:=\max \{ a_0 (s+1)^{2/3} , 16\sigma^2+3L^2 D^2 \}   $.
First, note that $Q\geq a_0 (s+1)^{2/3}$ and therefore $a_0\leq Q/(s+1)^{2/3}$ and the base step of the induction holds true. Now assume that the condition in \eqref{proof:bound_on_grad_sublinear_400} holds for $t=k-1$, i.e., 
\begin{equation}\label{proof:bound_on_grad_sublinear_500}
a_{k-1}\leq \frac{Q}{(k+s)^{2/3}}.
\end{equation}
The goal is to show that \eqref{proof:bound_on_grad_sublinear_400} also holds for $t=k$. To do so, first set $t=k$ in the expression in \eqref{proof:bound_on_grad_sublinear_300} to obtain 
\begin{align}\label{proof:bound_on_grad_sublinear_600}
a_{k}
\leq \left(1-\frac{2}{(k+s)^{2/3}}\right)a_{k-1}+ \frac{16\sigma^2+3L^2 D^2}{(k+s)^{4/3}}.
\end{align}
According to the definition of $Q$, we know that $Q\geq 16\sigma^2+3L^2 D^2$. Moreover, based on the induction hypothesis it holds that $a_{k-1}\leq \frac{Q}{(k+s)^{2/3}}$. Using these inequalities and the expression in \eqref{proof:bound_on_grad_sublinear_600} we can write 
\begin{align}\label{proof:bound_on_grad_sublinear_700}
a_{k}
\leq \left(1-\frac{2}{(k+s)^{2/3}}\right)\frac{Q}{(k+s)^{2/3}}+ \frac{Q}{(k+s)^{4/3}}.
\end{align}
Pulling out $\frac{Q}{(k+s)^{2/3}}$ as a common factor and simplifying and
reordering terms it follows that \eqref{proof:bound_on_grad_sublinear_700} is equivalent to
\begin{align}\label{proof:bound_on_grad_sublinear_800}
a_{k}
&\leq Q\left(\frac{(k+s)^{2/3}-1}{(k+s)^{4/3}}\right).
\end{align}
Based on the inequality 
\begin{align}\label{proof:bound_on_grad_sublinear_900}
((k+s)^{2/3}-1)((k+s)^{2/3}+1) < (k+s)^{4/3},
\end{align}
the result in \eqref{proof:bound_on_grad_sublinear_800} implies that 
\begin{align}\label{proof:bound_on_grad_sublinear_1000}
a_{k}
\leq \left(\frac{Q}{(k+s)^{2/3}+1}\right).
\end{align}
Since $(k+s)^{2/3}+1 \geq (k+s+1)^{2/3}$, the result in \eqref{proof:bound_on_grad_sublinear_1000} implies that 
\begin{align}\label{proof:bound_on_grad_sublinear_1100}
a_{k}
\leq \left(\frac{Q}{(k+s+1)^{2/3}}\right),
\end{align}
and the induction step is complete. Therefore, the result in \eqref{proof:bound_on_grad_sublinear_400} holds for all $t=0,\dots,T$. Indeed, by setting $s=8$, the claim in \eqref{eq:grad_error_bound_2} follows.


\subsection{How to Construct an Unbiased Estimator of the Gradient in Multilinear Extensions}\label{unbiased}
Recall that $\mbE_{\bbz\sim P} [\tilde{f}(S, \bbz)]$. In terms of the multilinear extensions, we obtain $F(\bbx) = \mbE_{\bbz\sim P} [\tilde{F}(\bbx, \bbz)]$, where $F$ and $ \tilde{F}$ denote the multilinear extension for $f$ and $\tilde{f}$, respectively. So $\nabla \tilde{F}(\bbx, \bbz)$ is an unbiased estimator of $\nabla F(\bbx)$ when $\bbz\sim P$. Note that $\tilde{F}(\bbx, \bbz)$ is a multilinear extension. 

It remains to provide an unbiased estimator for the gradient of a multilinear extension. We thus consider an arbitrary submodular set function $g$ with multilinear $G$. Our goal is to provide an unbiased estimator for $\nabla G(\bbx)$. We have $G(\bbx) = \sum_{S\subseteq V} \prod_{i\in S} x_i \prod_{j\not\in S} (1-x_j) g(S)$. Now, it can easily be shown that 
\begin{equation}
\frac{\partial G}{\partial x_i} = G(\bbx; x_i \leftarrow 1) - G(\bbx; x_i \leftarrow 0).
\end{equation}
where for example by $(\bbx; x_i \leftarrow 1)$ we mean a vector which has value $1$ on its $i$-th coordinate and is equal to $\bbx$ elsewhere. To create an unbiased estimator for $\frac{\partial G}{\partial x_i} $ at a point $\bbx$ we can simply sample a set $S$ by including each element in it independently with probability $x_i$ and use $g(S \cup \{i\}) - g(S \setminus \{i\})$ as an unbiased estimator for the $i$-th partial derivative. We can sample one single set $S$ and use the above trick for all the coordinates.  This involves $n$ function computations for $g$. Having a mini-batch size $B$ we can repeat this procedure $B$ times and then average.

\subsection{Proof of Lemma \ref{lemma:lip_constant}}\label{proof:lemma:lip_constant}
Based on the mean value theorem, we can write
\begin{align}
\nabla F(\bbx_t+\frac{1}{T}\bbv_t) -\nabla F(\bbx_T)  = \frac{1}{T} \bbH(\tbx_t)\bbv_t,
\end{align}
where $\tbx_t$ is a convex combination of $\bbx_t$ and $\bbx_t+\frac{1}{T}\bbv_t$ and $\bbH(\tbx_t):=\nabla^2 F(\tbx_t)$. This expression shows that the difference between the coordinates of the vectors $\nabla F(\bbx_t+\frac{1}{T}\bbv_t) $ and $\nabla F(\bbx_t)$ can be written as
\begin{align}
\nabla_{j} F(\bbx_t+\frac{1}{T}\bbv_t) -\nabla_{j} F(\bbx_t)  = \frac{1}{T} \sum_{i=1}^n H_{j,i}(\tbx_t)v_{i,t},
\end{align}
where $v_{i,t}$ is the $i$-th element of the vector $\bbv_t$ and $H_{j,i}$ denotes the component in the $j$-th row and $i$-th column of the matrix $\bbH$. Hence, the norm of the difference $|\nabla_{j} F(\bbx_t+\frac{1}{T}\bbv_t) -\nabla_{j} F(\bbx_t) |$ is bounded above by
\begin{align}
|\nabla_{j} F(\bbx_t+\frac{1}{T}\bbv_t) -\nabla_{j} F(\bbx_t) | \leq \frac{1}{T} \left|\sum_{i=1}^n H_{j,i}(\tbx_t)v_{i,t}\right|.
\end{align}
Note here that the elements of the matrix $\bbH(\tbx_t)$ are less than the maximum marginal value (i.e. $\max_{i,j} |H_{i,j}(\tbx_t)| \leq \max_{i \in \{1, \cdots, n\}} f(i) \triangleq m_f$). We thus get
\begin{align}
|\nabla_{j} F(\bbx_t+\frac{1}{T}\bbv_t) -\nabla_{j} F(\bbx_t) | \leq \frac{m_f}{T} \sum_{i=1}^n | v_{i,t}|.
\end{align}
Note that at each round $t$ of the algorithm, we have to pick a vector $\bbv_t \in \mathcal{C}$ s.t. the inner product $\langle \bbv_t, \bbd_t\rangle$ is maximized.  Hence, without loss of generality we can assume that the vector $\bbv_t$ is one of the extreme points of $\mathcal{C}$, i.e. it is of the form $1_{I}$ for some $I \in \mathcal{I}$ (note that we can easily force integer vectors).  Therefore by noticing that $\bbv_t$ is an integer vector with at most $r$ ones, we have
\begin{align}
|\nabla_{j} F(\bbx_t+\frac{1}{T}\bbv_t) -\nabla_{j} F(\bbx_t) | \leq \frac{m_f\sqrt{r}}{T} \sqrt{\sum_{i=1}^n | v_{i,t}|^2},
\end{align}
which yields the claim in \eqref{claim:lip_constant}.

\subsection{Proof of Theorem \ref{thm:multi_linear_extenstion_thm}}\label{proof:thm:multi_linear_extenstion_thm}

According to the Taylor's expansion of the function $F$ near the point $\bbx_t$ we can write 
\begin{align}\label{proof:final_result__multi_lin_100}
F(\bbx_{t+1}) &= F(\bbx_{t}) + \langle \nabla F(\bbx_t), \bbx_{t+1}-\bbx_t \rangle 
 + \frac{1}{2}\langle \bbx_{t+1}-\bbx_t, \bbH(\tbx_t) (\bbx_{t+1}-\bbx_t)\rangle \nonumber\\
&= F(\bbx_{t}) + \frac{1}{T}\langle \nabla F(\bbx_t), \bbv_t \rangle 
+ \frac{1}{2T^2}\langle \bbv_t, \bbH(\tbx_t) \bbv_t\rangle,
 \end{align}
where $\tbx_t$ is a convex combination of $\bbx_t$ and $\bbx_t+\frac{1}{T}\bbv_t$ and $\bbH(\tbx_t):=\nabla^2 F(\tbx_t)$. Note that based on the inequality $\max_{i,j} |H_{i,j}(\tbx_t)| \leq \max_{i \in \{1, \cdots, n\}} f(i) \triangleq m_f$, we can lower bound $H_{ij}$ by $-m_f$. Therefore, 
\begin{align}\label{proof:final_result__multi_lin_200}
\langle \bbv_t, \bbH(\tbx_t) \bbv_t\rangle 
= \sum_{j=1}^n\sum_{i=1}^n v_{i,t} v_{j,t} H_{ij}(\tbx_t)\geq -m_f\sum_{j=1}^n\sum_{i=1}^n v_{i,t} v_{j,t} =-m_f\left(\sum_{i=1}^n v_{i,t}\right)^2=-m_f r\|\bbv_t\|^2,
\end{align}
where the last inequality is because $\bbv_t$ is a vector with $r$ ones and $n-r$ zeros (see the explanation in the proof of Lemma~\ref{lemma:lip_constant}).
Replace the expression $\langle \bbv_t, \bbH(\tbx_t) \bbv_t\rangle $ in \eqref{proof:final_result__multi_lin_100} by its lower bound in \eqref{proof:final_result__multi_lin_200} to obtain
\begin{align}\label{proof:final_result__multi_lin_300}
F(\bbx_{t+1}) \geq F(\bbx_{t}) + \frac{1}{T}\langle \nabla F(\bbx_t), \bbv_t \rangle 
- \frac{m_f r}{2T^2}\|\bbv_t\|^2.
\end{align}
In the following lemma we derive a variant of the result in Lemma \ref{lemma:bound_on_grad_approx_sublinear} for the multilinear extension setting.

\begin{lemma}\label{lemma:sub_grad_error}
Consider \alg (SCG)  outlined in Algorithm~\ref{algo_SCGGA}, and recall the definitions of the function $F$ in  \eqref{eq:def_multi_linear_extension}, the rank $r$, and $m_f \triangleq \max_{i \in \{1, \cdots, n\}} f(i)$.  If we set $\rho_t=\frac{4}{(t+8)^{2/3}}$, then for $t=0,\dots,T$ and for $j=1,\dots,n$ it holds
\begin{align}\label{ali_mansoor}
\E{ \|\nabla F(\bbx_{t}) - \bbd_{t}|^2}&\leq \frac{Q}{(t+9)^{2/3}},
\end{align}
 where  $Q:=\max \{ 9^{2/3}\|\nabla F(\bbx_{0}) - \bbd_{0}\|^2  , 16\sigma^2+3m_f^2 r D^2 \}$. 
\end{lemma}

\begin{proof}
The proof is similar to the proof of Lemma~\ref{lemma:bound_on_grad_approx_greedy}. The main difference is to write the analysis for the $j$-th coordinate and replace and $L$ by  ${m_f\sqrt{r}}$ as shown in Lemma \ref{lemma:lip_constant}. Then using the proof techniques in Lemma \ref{lemma:bound_on_grad_approx_sublinear} the claim in Lemma \ref{lemma:sub_grad_error} follows.
\end{proof}

The rest of the proof is identical to the proof of Theorem 1, by following the steps from \eqref{proof:final_result_100} to \eqref{proof:final_result_900} and considering the bound in \eqref{ali_mansoor} we obtain
\begin{equation}\label{proof:final_result_900_v2}
 \E{F(\bbx_{T}) }\geq (1- 1/e) F(\bbx^*)  - \frac{2DQ^{1/2}}{T^{1/3}}-  \frac{m_f rD^2}{2T},
\end{equation}
where $Q:=\max \{ \|\nabla F(\bbx_{0}) - \bbd_{0}\|^2 9^{2/3} , 16\sigma^2+3rm_f^2 D^2 \}$. Therefore, the claim in Theorem \ref{thm:multi_linear_extenstion_thm} follows.

\bibliography{bibliography}
\bibliographystyle{abbrvnat}

\end{document}